\documentclass[12pt,reqno,a4paper]{amsart}
\usepackage{CJK,latexsym,amsmath,amsfonts,amssymb,amsthm,graphicx}
\theoremstyle{plain}
\newtheorem{Thm}{Theorem}
\newtheorem{Lem}{Lemma}

\theoremstyle{definition}

\theoremstyle{remark}

\def\pmod #1{\ ({\rm{mod}}\ #1)}
\def\Z{\mathbb Z}
\def\N{\mathbb N}

\def\qbinom #1#2#3{{\genfrac{[}{]}{0pt}{}{#1}{#2}}_{#3}}

\def\pmod #1{\ ({\rm{mod}}\ #1)}

\def\1{{\bf{1}}}

\def\x{{\bf{x}}}
\def\y{{\bf{y}}}
\def\l{{\mathfrak{l}}}
\def\k{{\mathfrak{k}}}
\def\h{{\mathfrak{h}}}
\def\e{{\mathfrak{e}}}
\def\v{{\mathfrak{v}}}
\def\O{{\mathcal{O}}}
\def\cQ{{\mathcal{Q}}}
\def\cP{{\mathcal{P}}}

\begin{document}
\title{A Lucas-type congruence for $q$-Delannoy numbers}
\author{Hao Pan}
\email{haopan79@zoho.com}
\address{Department of Mathematics, Nanjing University,
Nanjing 210093, People's Republic of China}
\maketitle

\section{Introduction}
\setcounter{Lem}{0}\setcounter{Thm}{0}\setcounter{Cor}{0}
\setcounter{equation}{0}

In combinatorics, the Delannoy number $D(h,k)$ denotes the number of lattices paths from $(0,0)$ to $(h,k)$, by only using three kinds of steps: east $(1,0)$, north $(0,1)$ and northeast $(1,1)$. 
The Delannoy numbers satisfy the recurrence relation
\begin{equation}\label{recurd}
D(h+1,k+1)=D(h+1,k)+D(h,k+1)+D(h,k).
\end{equation}
They also have two closed-form expressions:
\begin{align}
D(h,k)=&
\sum_{j=0}^h\binom{k}{j}\binom{h+k-j}{k}\\
=&\label{twod}\sum_{j=0}^h2^j\binom{k}{j}\binom{h}{j}.
\end{align}
And the generating function of $D(h,k)$ is
\begin{equation}
\frac{1}{1-x-y-xy}=\sum_{h,k\geq 0}D(h,k)x^hy^k.
\end{equation}

On the other hand, the Delannoy numbers also have some interesting arithmetic properties. 
The well-known Lucas congruence says that for any prime $p$,
\begin{equation}\label{lucas}
\binom{ap+b}{cp+d}\equiv\binom{a}{c}\binom{b}{d}\pmod{p},
\end{equation}
where $a,b,c,d\in\Z$ and $0\leq b,d\leq p-1$. In fact, we also have the following Lucas-type congruence for $D(h,k)$:
\begin{equation}\label{lucasd}
D(ap+b,cp+d)\equiv D(a,c)D(b,d)\pmod{p}.
\end{equation}
In fact, (\ref{lucasd}) is a  special case of Theorem 1 of \cite{Razpet02} (by substituting $a=b=c=1$).

The $q$-binomial coefficient is defined by 
$$
\qbinom{h}{k}q=\frac{[h]_q[h-1]_q\cdots[h-k+1]_q}{[k]_q[k-1]_q\cdot[1]_q},
$$
where the $q$-integer $[n]_q=(1-q^n)/(1-q)$.
In particular, define $\qbinom{h}{0}q=1$ and $\qbinom{h}{k}q=0$ provided $k<0$.
Let $\Phi_n(q)$ be the $n$-th cyclotomic polynomial. There is a $q$-analogue of the Lucas congruence (cf. \cite[Theorem 2.2]{Sagan92}):
\begin{equation}\label{qlucas}
\qbinom{an+b}{cn+d}q\equiv\binom{a}{c}\qbinom{b}{d}q\pmod{\Phi_n(q)},
\end{equation}
where $a,b,c,d\in\N$ with $0\leq b,d\leq n-1$, and the above congruence is considered over the polynomial ring $\Z[q]$.
Noting that $\Phi_n(q)=[n]_q$ when $n$ is prime, (\ref{lucas}) follows from (\ref{qlucas}) by substituting $q=1$.

It is natural to ask whether we also can give a $q$-analogue of (\ref{lucasd}). Define the $q$-Delannoy number
\begin{equation}
D_q(n,k)=\sum_{j=0}^nq^{\binom{j+1}2}\qbinom{k}{j}q\qbinom{n+k-j}{k}q.
\end{equation}
In particular, set $D_q(n,k)=0$ if either $h$ or $k$ is negative.
As we shall see later, $D_q(n,k)$ is a suitable $q$-analogue of the Delannoy numbers, since the $q$-analogues of (\ref{recurd}) and (\ref{twod}) also can be established for $D_q(n,k)$.
However, it seems that such a $q$-analogue is not described in any literature, though the $q$-analogues
of the Schr\"oder number, which is closely related to the Delannoy number, have been studied in several papers \cite{BSS93, LiuWang07, ITZ09}.

In this short paper, we shall prove
\begin{Thm}\label{qlucast1} Suppose that $n$ is odd. Then
\begin{equation}\label{qlucasdo1}
D_q(an+b,cn+d)\equiv D(a,c)D_q(b,d)\pmod{\Phi_n(q)},
\end{equation}
where $a,b,c,d\in\N$ and $0\leq b,d\leq n-1$.
And if $n$ is even, then
\begin{equation}\label{qlucasde1}
D_q(an+b,cn+d)\equiv D_q(b,d)\pmod{\Phi_n(q)}.
\end{equation}
\end{Thm}
Using the recurrence relation (\ref{recurd}), we can easily reduce Theorem \ref{qlucast1} to
\begin{Thm}\label{qlucast2} Suppose that $n$ is odd. Then
\begin{equation}\label{qlucasdo2}
D_q(h+n,k+n)\equiv D_q(h+n,k)+D_q(h,k+n)+D_q(h,k)
\pmod{\Phi_n(q)}.
\end{equation}
And if $n$ is even, then
\begin{equation}\label{qlucasde2}
D_q(h+n,k+n)\equiv D_q(h+n,k)+D_q(h,k+n)-D_q(h,k)
\pmod{\Phi_n(q)}.
\end{equation}
\end{Thm}
In fact, clearly (\ref{qlucasdo1}) holds when $a=c=0$. Thus using (\ref{qlucasdo2}) and an induction on $a+c$, we can get that
for odd $n$,
\begin{align*}
&D_q((a+1)n+b,(c+1)n+d)\\\equiv& 
D_q((a+1)n+b,cn+d)+D_q(an+b,(c+1)n+d)+D_q(an+b,cn+d)\\
\equiv&\big(D(a+1,c)+D(a,c+1)+D(a,c)\big)D_q(b,d)
=D(a+1,c+1)D_q(b,d)\pmod{\Phi_n(q)}.
\end{align*}
And (\ref{qlucasde1}) can be similarly derived from (\ref{qlucasde2}).

In the next section, we shall discuss some basic combinatorial properties of $D_q(h,k)$. And 
Theorem \ref{qlucast2} will be proved in the third section. Our proof of Theorem \ref{qlucast2} is combinatorial and
bases on the method of group actions, which was developed by Rota and Sagan \cite{RotaSagan80,Sagan85,Sagan92}. 

\section{The combinatorics of $q$-Delannoy numbers}
\setcounter{Lem}{0}\setcounter{Thm}{0}\setcounter{Cor}{0}
\setcounter{equation}{0}

In view of the $q$-binomial theorem \cite[Corollary 10.2.2(c)]{AAR99} and the $q$-Chu-Vandermonde identity \cite[Exercise 10.4(b)]{AAR99}, we have
\begin{align*}
&\sum_{j=0}^hq^{(h-j)(k-j)}(-q;q)_j\qbinom{k}{j}q\qbinom{h}{j}q=
\sum_{j=0}^h\qbinom{k}{j}q\qbinom{h}{j}q\sum_{i=0}^jq^{\binom{i+1}2}\qbinom{j}{i}q\\
=&
\sum_{i=0}^hq^{\binom{i+1}2}\qbinom{k}{i}q\sum_{j=i}^nq^{(h-j)(k-j)}\qbinom{h}{h-j}q\qbinom{k-i}{j-i}q=
\sum_{i=0}^hq^{\binom{i+1}2}\qbinom{k}{i}q\qbinom{n+k-i}{h-i}q.
\end{align*}
Thus we get the following $q$-analogue of (\ref{twod}):
\begin{equation}
D_q(h,k)=\sum_{j=0}^nq^{(h-j)(k-j)}(-q;q)_j\qbinom{k}{j}q\qbinom{h}{j}q.
\end{equation}
Furthermore, $D_q(h,k)$ also satisfies the recurrence relation
\begin{equation}
D_q(h+1,k+1)=D_q(h+1,k)+q^{k+1}D_q(h,k+1)+q^{k+1}D_q(h,k).
\end{equation}
In fact, since
$$
\qbinom{h}{k}q=q^k\qbinom{h-1}{k}q+\qbinom{h-1}{k-1}q=\qbinom{h-1}{k}q+q^{h-k}\qbinom{h-1}{k-1}q,
$$
we have
\begin{align*}
&D_q(h,k)=\sum_{j=0}^hq^{\binom{j+1}2}\qbinom{k}{j}q\qbinom{h+k-j}{k}q\\
=&
\sum_{j=0}^hq^{\binom{j+1}2}\qbinom{k}{j}q\bigg(q^k\qbinom{h-1+k-j}{k}q+\qbinom{h+k-j-1}{k-1}q\bigg)\\
=&q^kD_q(h-1,k)+\sum_{j=0}^hq^{\binom{j+1}2}\bigg(\qbinom{k-1}{j}q+q^{k-j}\qbinom{k-1}{j-1}q\bigg)\qbinom{h+k-j-1}{k-1}q\\
=&D_q(h,k-1)+q^kD_q(h-1,k)+q^kD_q(h-1,k-1).
\end{align*}

Below we shall give a combinatorial interpretation for $D_q(h,k)$.
For a step $\vec{e}\in\{(1,0),(0,1),(1,1)\}$, let $\x(\vec{e})$ and $\y(\vec{e})$ denote the $x$-coordinate and $y$-coordinate of $\vec{e}$ respectively. 
We may write a lattice path $\l$ as $\l=\langle\vec{e}_1,\vec{e}_2,\ldots,\vec{e}_m\rangle$, provided that
$\l$ uses the steps $\vec{e}_1,\vec{e}_2,\ldots,\vec{e}_m$ successively. 

For a path $\l=\langle\vec{e}_1,\ldots,\vec{e}_m\rangle$ define
$$
{\bf x}(\l)=\sum_{j=1}^m{\bf x}(\vec{e_j}),\qquad {\bf y}(\l)=\sum_{j=1}^m{\bf y}(\vec{e_j}).
$$
That is, ${\bf x}(\l)$ (resp. ${\bf y}_j(\l)$) is the difference between the $x$-coordinates (resp. $y$-coordinates)
of the endpoint and the start point of $\l$. Furthermore, define
$$
\sigma(\l)=\sum_{\substack{1\leq j\leq m\\ {\bf y}(\vec{e_j})=1}}{\bf x}(\l_j),
$$
where $\l_j=\langle\vec{e}_1,\ldots,\vec{e}_j\rangle$.
Thus if $\l$ is a path from $(0,0)$ to $(h,k)$, then $\sigma(\l)$ is the sum
of the $x$-coordinates of the endpoints of those steps of $\l$ whose $y$-coordinate is $1$.

Let $\cP_{h,k}$ be the set of all lattice paths from $(0,0)$ to $(h,k)$ using the steps $(1,0)$, $(0,1)$ and $(1,1)$.
The following result gives a combinatorial interpretation of $D_q(h,k)$.
\begin{Thm}
\begin{equation}\label{combind}
D_q(h,k)=\sum_{\l\in \cP_{h,k}}q^{\sigma(\l)}.
\end{equation}
\end{Thm}
\begin{proof} We shall use induction on $h+k$.
There is nothing to do when $h=k=0$. Assume that $h+k\geq 1$ and the assertion holds when $h+k$ is smaller. It suffices to verify the right side of (\ref{combind}) also satisfies the same recurrence relation as $D_q(h,k)$.

Assume that $\l\in\cP_{h,k}$ and $\l=\langle\vec{e}_1,\ldots,\vec{e}_m\rangle$. 
Let $\l_*=\langle\vec{e}_1,\ldots,\vec{e}_{m-1}\rangle$. Since the endpoint of $\vec{e}_{m}$ is $(h,k)$,
$$
\sigma(\l)=\begin{cases}\sigma(\l_*),&\text{if }\vec{e}_{m}=(1,0),\\
\sigma(\l_*)+k,&\text{if }\vec{e}_{m}=(0,1)\text{ or }(1,1).
\end{cases}
$$
So
\begin{align*}
\sum_{\l\in \cP_{h,k}}q^{\sigma(\l)}=
&\sum_{\substack{\l\in \cP_{h,k}\\ \vec{e}_\l=(1,0)}}q^{\sigma(\l)}+
\sum_{\substack{\l\in \cP_{h,k}\\ \vec{e}_\l=(0,1)}}q^{\sigma(\l)}+
\sum_{\substack{\l\in \cP_{h,k}\\ \vec{e}_\l=(1,1)}}q^{\sigma(\l)}\\
=&\sum_{\substack{\l_*\in \cP_{h,k-1}}}q^{\sigma(\l_*)}+
q^k\sum_{\substack{\l_*\in \cP_{h-1,k}}}q^{\sigma(\l_*)}+
q^k\sum_{\substack{\l_*\in \cP_{h-1,k-1}}}q^{\sigma(\l_*)}\\
=&D_q(h,k-1)+q^kD_q(h-1,k)+q^{k}D_q(h-1,k-1).
\end{align*}
\end{proof}
For two paths $\l_1$ and $\l_2$, let $\l_1+\l_2$ denote the path obtained by moving the start p.
That is, if $\l_1=\langle\vec{e}_1,\ldots,\vec{e}_m\rangle$ and $\l_2=\langle\vec{f}_1,\ldots,\vec{f}_n\rangle$, then
$\l_1+\l_2=\langle\vec{e}_1,\ldots,\vec{e}_m,\vec{f}_1,\ldots,\vec{f}_n\rangle$.
Clearly 
$$
{\bf x}(\l_1+\l_2)={\bf x}(\l_1)+{\bf x}(\l_2),\qquad
{\bf y}(\l_1+\l_2)={\bf y}(\l_1)+{\bf y}(\l_2).
$$
And for $\sigma(\l_1+\l_2)$, we also have
\begin{Lem}\label{sigmaadd}
$$
\sigma(\l_1+\l_2)=\sigma(\l_1)+\sigma(\l_2)+{\bf x}(\l_1){\bf y}(\l_2).
$$
\end{Lem}
\begin{proof}
Assume that $\l_1=\langle\vec{e}_1,\ldots,\vec{e}_m\rangle$ and $\l_2=\langle\vec{f}_1,\ldots,\vec{f}_n\rangle$. Then
$$
\sigma(\l_1)=\sum_{\substack{1\leq j\leq m\\ {\bf y}(\vec{e_j})=1}}\sum_{i=1}^j{\bf x}(\vec{e}_j),\qquad
\sigma(\l_2)=\sum_{\substack{1\leq j\leq n\\ {\bf y}(\vec{f_j})=1}}\sum_{i=1}^j{\bf x}(\vec{f}_j).
$$
So
\begin{align*}
&\sigma(\l_1+\l_2)=\sigma(\langle\vec{e}_1,\ldots,\vec{e}_m,\vec{f}_1,\ldots,\vec{f}_n\rangle)\\
=&\sum_{\substack{1\leq j\leq m\\ {\bf y}(\vec{e_j})=1}}\sum_{i=1}^j{\bf x}(\vec{e}_j)+
\sum_{\substack{1\leq j\leq n\\ {\bf y}(\vec{f_j})=1}}\bigg(\sum_{i=1}^n{\bf x}(\vec{e}_j)+\sum_{i=1}^j{\bf x}(\vec{f}_j)\bigg)\\
=&\sigma(\l_1)+\sigma(\l_2)+\sum_{\substack{1\leq j\leq n\\ {\bf y}(\vec{f_j})=1}}\sum_{i=1}^n{\bf x}(\vec{e}_j)
=\sigma(\l_1)+\sigma(\l_2)+{\bf y}(\l_2){\bf x}(\l_1).
\end{align*}
\end{proof}
Lemma \ref{sigmaadd} will be used in our proof of Theorem \ref{qlucast2}.

\section{Lucas' type congruence for $q$-Delannoy numbers}
\setcounter{Lem}{0}\setcounter{Thm}{0}\setcounter{Cor}{0}
\setcounter{equation}{0}

First, we shall partition $\cP_{h+n,k+n}$ into several subsets. 
Let $L_1$ be the vertical line from $(h,k)$ to $(h+n,k)$ and $L_2$ be the horizontal line from $(h,k)$ to $(h,k+n)$. 
For $\l\in P_{h+n,k+n}$, let $\bar{\l}$ denote the part of $\l$ on the $L_1\cup L_2$. Of course, maybe $\bar{\l}$ just contains one point, i.e., the start point of $\bar{\l}$ coincides with the end point.
Let $\check{\l}$ be the part of $\l$ from the origin to the start point of $\bar{\l}$, and let $\hat{\l}$ be 
the part of $\l$ the end point of $\bar{\l}$ to $(h+n,h+k)$. The following graph shows a concrete examples for $\check{\l}$,
$\bar{\l}$ and $\hat{\l}$.
\begin{center}
\includegraphics[width=0.3\textwidth]{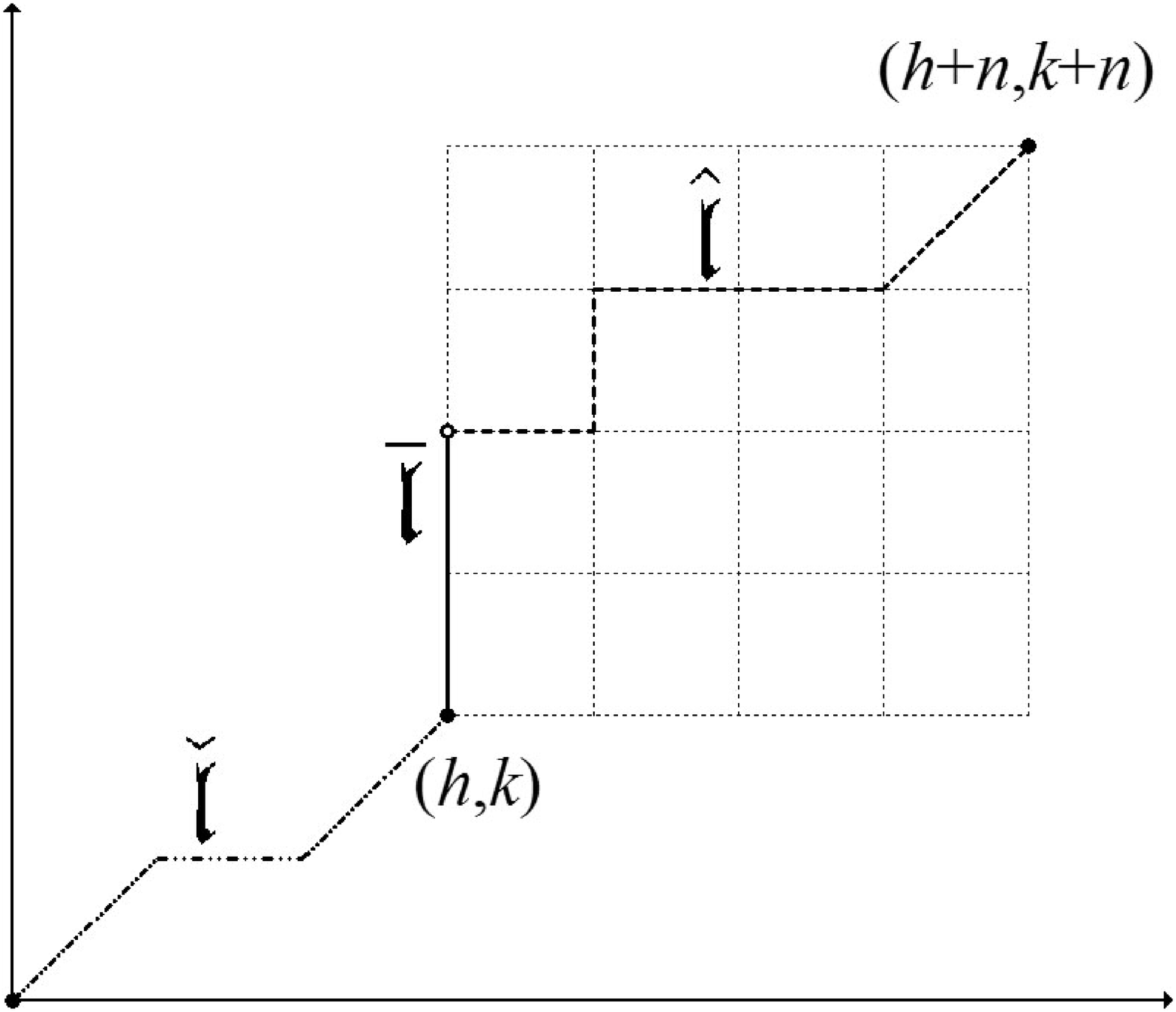}
\end{center}
Let
$$
\cQ_1=\{\l\in\cP_{h+n,k+n}:\,\text{ the end point of }\bar{\l}\text{ is }(h+i,k)\text{ for some }1\leq i\leq n\}
$$
and
$$
\cQ_2=\{\l\in\cP_{h+n,k+n}:\,\text{ the start point of }\bar{\l}\text{ is }(h,k+i)\text{ for some }1\leq i\leq n\}.
$$
Clearly if $\l\in\cP_{h+n,k+n}$ but $\l\not\in \cQ_1\cup \cQ_2$, then $\l$ must touch $(h,k)$, and the step of $\l$,
whose start point is $(h,k)$, must be $(1,0)$ or  $(1,1)$. Let
$$
\cQ_3=\{\l\in P_{h+n,k+n}:\,\l\not\in \cQ_1\cup \cQ_2\text{ and no step of }\hat{\l}\text{ is }(1,1)\}
$$
and
$$
\cQ_4=\{\l\in P_{h+n,k+n}:\,\l\not\in \cQ_1\cup \cQ_2\text{ and at least one step of }\hat{\l}\text{ is }(1,1)\}.
$$
Then $\cP_{h+n,k+n}=\cQ_1\cup\cQ_2\cup\cQ_3\cup\cQ_4$.  The following graph gives the examples of $\cQ_1$, $\cQ_2$, $\cQ_3$ and $\cQ_4$.
\begin{center}
\includegraphics[width=\textwidth]{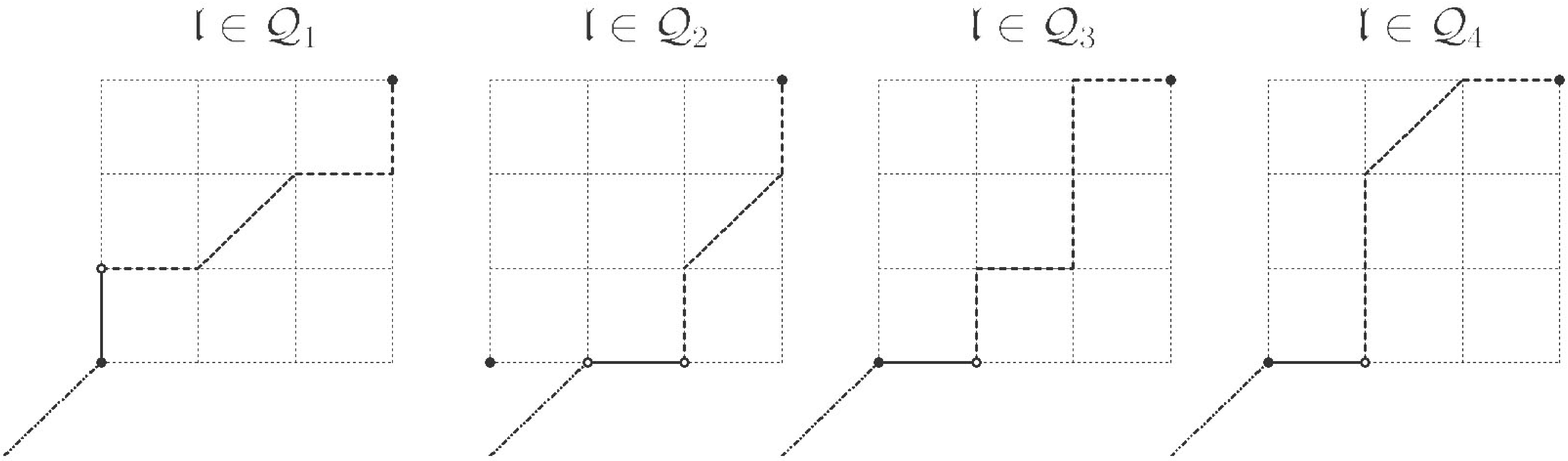}
\end{center}

Let $\Z_n=\Z/n\Z$. We need to define a group action of $\Z_n$ on $\cQ_1$. 
For convenience, we write $\Z_n=\{1,2,\ldots,n\}$. We may assume that $$
\hat{\l}=\k_1+\k_2+\cdots+\k_n,$$
where for each $j$, the first step of $\k_j$ is $(1,0)$ or $(1,1)$ and the other steps are $(0,1)$. 
An example is given as follows.
\begin{center}
\includegraphics[width=0.3\textwidth]{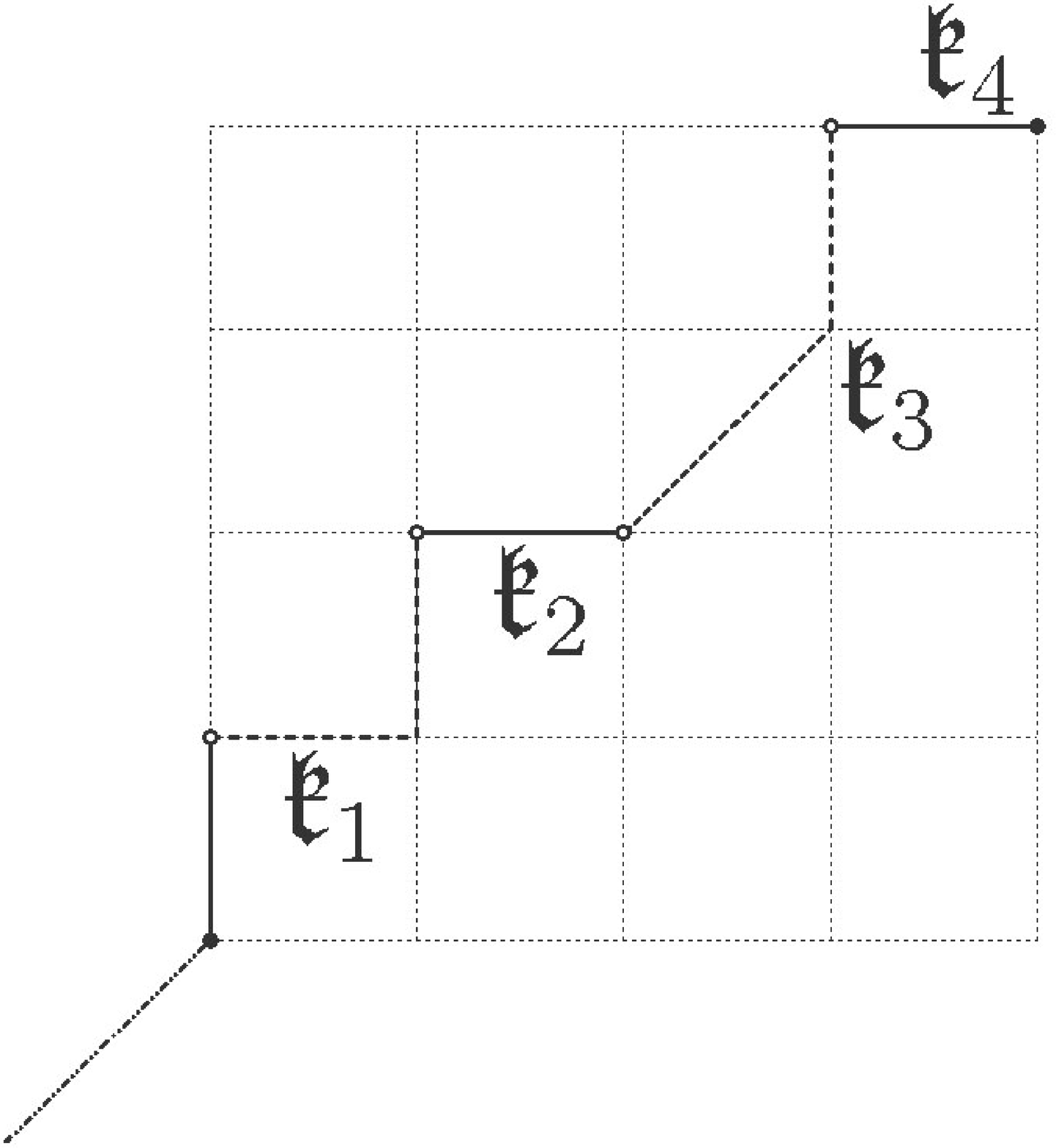}
\end{center}
Define
$$
\phi_1(\l)=\check{\l}+\bar{\l}+(\k_n+\k_1+\cdots+\k_{n-1}).
$$
The following graph shows a concrete transformation of $\phi_1$.
\begin{center}
\includegraphics[width=0.6\textwidth]{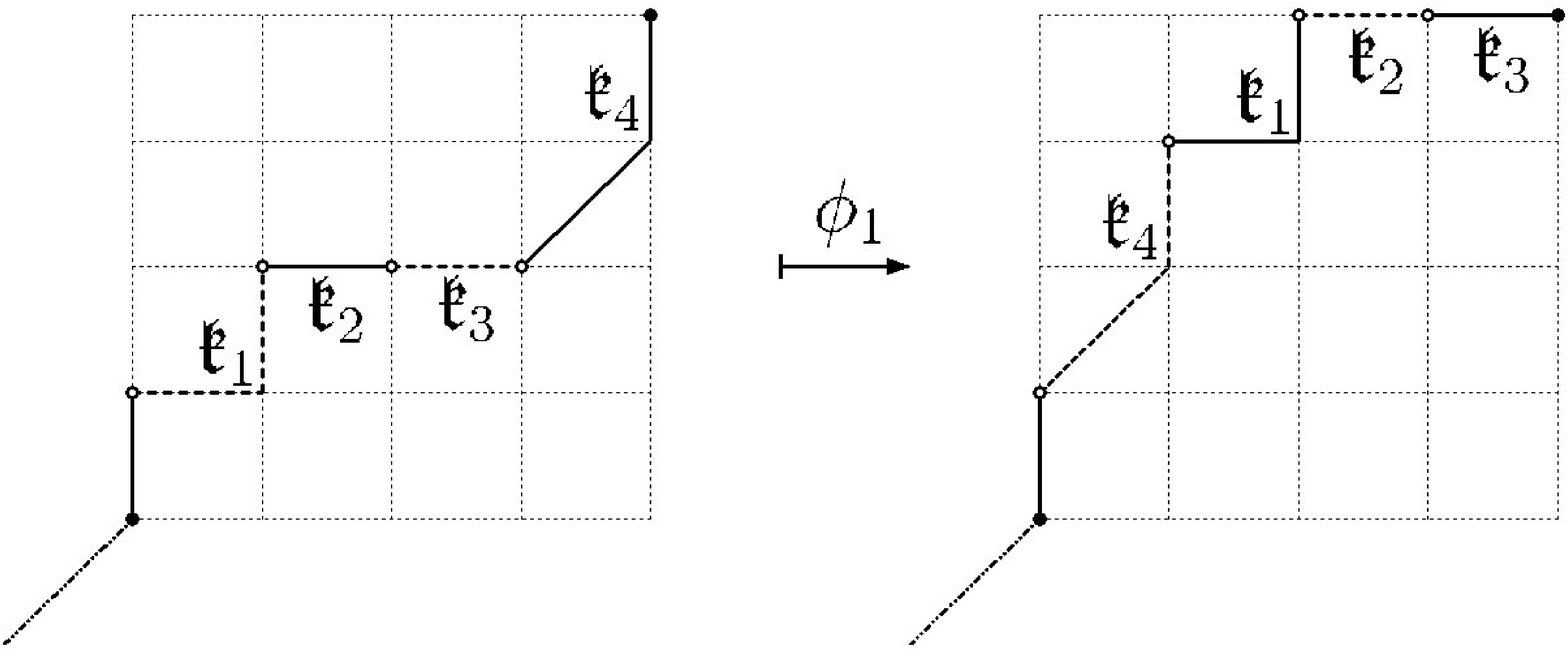}
\end{center}
And for each $j\in\Z_n$, let $\phi_j$ be the $j$-th iteration of $\phi_1$. Clearly $\phi$ is a group action of $\Z_n$ on $\cQ_1$.
For $\l\in \cQ_1$, let
$$
\O_{\phi,\l}=\{\phi_j(\l):\,1\leq j\leq n\}.
$$ 
If
$|\O_{\phi,\l}|=d$, then we know that $d$ is the least divisor of $n$ such that $\phi_d(\l)=\l$. Furthermore, $\hat{\l}$ has the form
$$
\hat{\l}=(\k_1+\cdots+\k_d)+(\k_1+\cdots+\k_d)+\cdots+(\k_1+\cdots+\k_d).$$
Clearly $|\O_{\phi,\l}|=1$ if and only if ${\bf y}(\hat{\l})=0$, i.e., all steps of $\hat{\l}$ is $(1,0)$.

Assume that $|\O_{\phi,\l}|>1$, i.e., $\l$ isn't the fixed point of $\phi_1$. We shall prove that 
\begin{equation}\label{hOphi}
\sum_{\h\in\O_{\phi,\l}}q^{\sigma(\h)}\equiv 0\pmod{\Phi_n(q)}.
\end{equation}
By Lemma \ref{sigmaadd}, we have
\begin{align}\label{sigmal}
&\sigma(\l)=\sigma(\check{\l}+\bar{\l}+\hat{\l})\notag\\
=&\sigma(\check{\l}+\bar{\l})+\frac{n}{d}\cdot\sigma(\k_1+\cdots+\k_d)+{\bf y}(\k_1+\cdots+\k_d)\sum_{j=0}^{\frac{n}d-1}({\bf x}(\check{\l}+\bar{\l})+j{\bf x}(\k_1+\cdots+\k_d))\notag\\
=&\sigma(\check{\l}+\bar{\l})+\frac{n\sigma(\k_1+\cdots+\k_d)}{d}+{\bf y}(\k_1+\cdots+\k_d)\bigg(\frac{n{\bf x}(\check{\l}+\bar{\l})}{d}+\binom{\frac nd}{2}{\bf x}(\k_1+\cdots+\k_d)\bigg).
\end{align}
Notice that
\begin{align*}
\sigma(\k_1+\cdots+\k_d)=&\sigma(\k_1+\cdots+\k_{d-1})+\sigma(\k_d)+{\bf x}(\k_1+\cdots+\k_{d-1}){\bf y}(\k_d)\\
=&\sigma(\k_1+\cdots+\k_{d-1})+\sigma(\k_d)+(d-1){\bf y}(\k_d)
\end{align*}
and
\begin{align*}
\sigma(\k_d+\k_1+\cdots+\k_{d-1})=&\sigma(\k_d)+\sigma(\k_1+\cdots+\k_{d-1})+{\bf x}(\k_d){\bf y}(\k_1+\cdots+\k_{d-1})\\
=&\sigma(\k_d)+\sigma(\k_1+\cdots+\k_{d-1})+{\bf y}(\k_1+\cdots+\k_{d-1}).
\end{align*}
So by (\ref{sigmal}),
\begin{align*}
\sigma(\phi_1(\l))-\sigma(\l)=&\frac{n}{d}({\bf y}(\k_1+\cdots+\k_{d-1})+{\bf y}(\k_d)-d{\bf y}(\k_d))\\
\equiv&\frac{n}{d}\cdot{\bf y}(\k_1+\cdots+\k_{d})\pmod{n}.
\end{align*}
Note that $$
\frac{n}{d}\cdot{\bf y}(\k_1+\cdots+\k_{d})={\bf y}(\hat{\l})<n,
$$
since the start point of $\hat{\l}$ is $(h+i,k)$ for some $i\geq 1$.
Thus we get
\begin{align*}
\sum_{\h\in \O_\l}q^{\sigma(\l)}=q^{\sigma(\l)}+\sum_{j=1}^{d-1}q^{\sigma(\phi_j(\l))}\equiv&
q^{\sigma(\l)}\sum_{j=0}^{d-1}q^{j\cdot\frac{n}{d}\cdot{\bf y}(\k_1+\cdots+\k_{d})}\\
=&
q^{\sigma(\l)}\cdot\frac{1-q^{d\cdot\frac{n}{d}\cdot{\bf y}(\k_1+\cdots+\k_{d})}}{1-q^{\frac{n}{d}\cdot{\bf y}(\k_1+\cdots+\k_{d})}}\equiv0\pmod{\Phi_n(q)}.
\end{align*}
Note that $\cQ_1$ can be partitioned into the union of the orbits of some $\l\in \cQ_1$. It follows from (\ref{hOphi}) that
$$
\sum_{\l\in \cQ_1}q^{\sigma(\l)}\equiv \sum_{\substack{\l\in \cQ_1\\ {\bf y}(\hat{\l})=0}}q^{\sigma(\l)}\pmod{\Phi_n(q)}.
$$

For $\l\in \cQ_2$, we may also write
$$
\hat{\l}=\k_1+\cdots+\k_n,
$$
where the first steps of $\k_1,\ldots,\k_n$ are $(0,1)$ or $(1,1)$, and the other steps of $\k_1,\ldots,\k_n$ are $(1,0)$.
There is an example as follows.
\begin{center}
\includegraphics[width=0.3\textwidth]{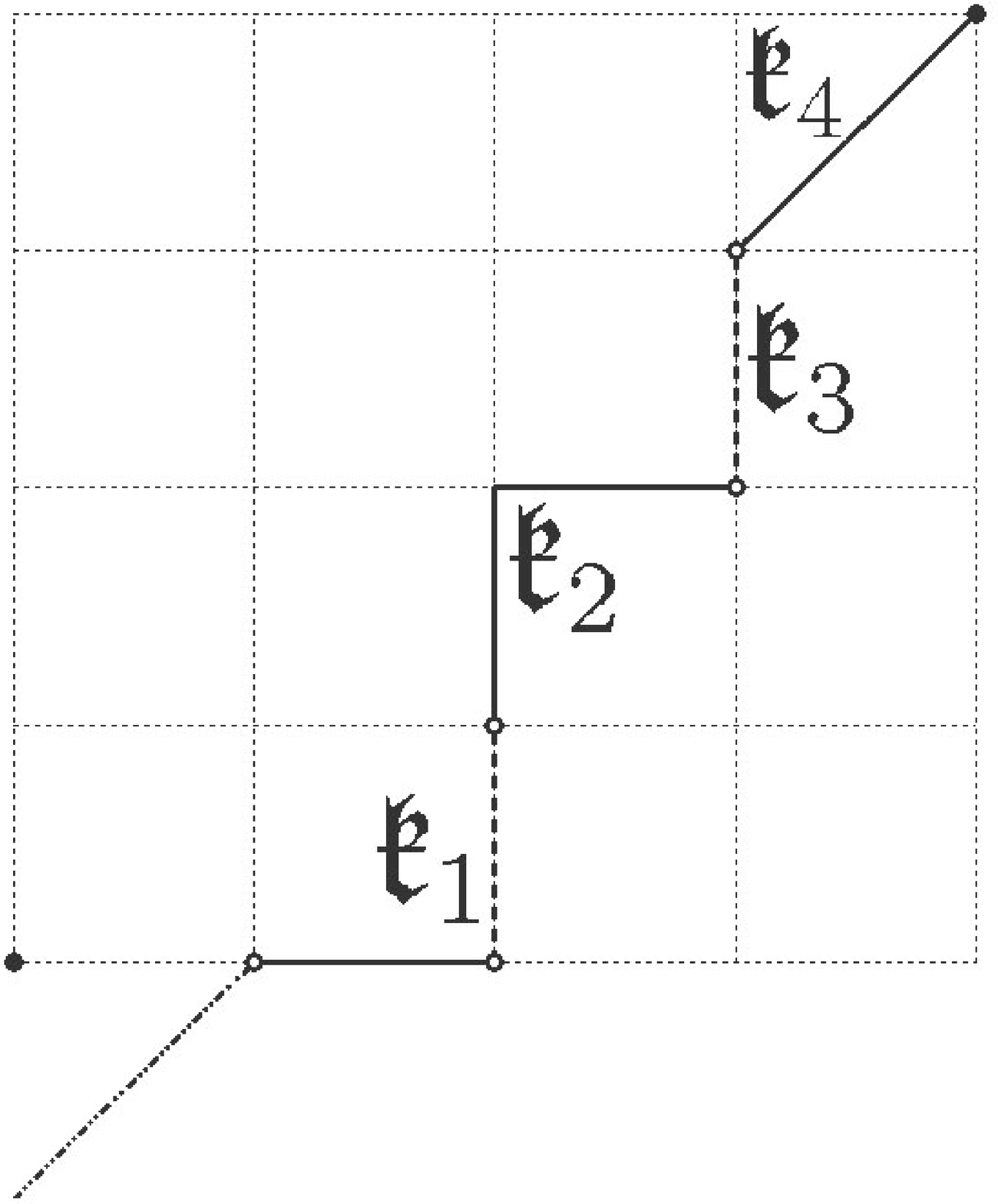}
\end{center}
Define
$$
\psi_1(\l)=\check{\l}+\bar{\l}+(\k_n+\k_1+\cdots+\k_{n-1}).
$$
Then $\psi_1$ can generate a group of $\Z_n$ on $\cQ_2$.
Assume that $|\O(\psi,\l)|=d$ where $\O(\psi,\l)=\{\psi_j(\l):\,j\in\Z_n\}.$, i.e.,
$$
\hat{\l}=(\k_1+\cdots+\k_d)+\cdots+(\k_1+\cdots+\k_d).
$$
Clearly
\begin{align*}
\sigma(\k_1+\cdots+\k_d)=&\sigma(\k_1+\cdots+\k_{d-1})+\sigma(\k_d)+{\bf x}(\k_1+\cdots+\k_{d-1}){\bf y}(\k_d)\\
=&\sigma(\k_1+\cdots+\k_{d-1})+\sigma(\k_d)+{\bf x}(\k_1+\cdots+\k_{d-1}).
\end{align*}
and
\begin{align*}
\sigma(\k_d+\k_1+\cdots+\k_{d-1})=&\sigma(\k_d)+\sigma(\k_1+\cdots+\k_{d-1})+{\bf x}(\k_d){\bf y}(\k_1+\cdots+\k_{d-1})\\
=&\sigma(\k_d)+\sigma(\k_1+\cdots+\k_{d-1})+{\bf x}(\k_d)(d-1).
\end{align*}
By (\ref{sigmal}), we also have
$$
\sigma(\psi_1(\l))-\sigma(\l)=\sigma(\k_d+\k_1+\cdots+\k_{d-1})-\sigma(\k_1+\cdots+\k_{d})\equiv
(d-1){\bf x}(\k_1+\cdots+\k_{d})\pmod{\Phi_n(q)}.
$$
Now
$$
\frac{n}{d}\cdot{\bf x}(\k_1+\cdots+\k_{d})={\bf x}(\hat{\l})<n.
$$
So if $d>1$, i.e., ${\bf x}(\hat{\l})\neq 0$, then
\begin{align*}
\sum_{\h\in\O_{\psi,\l}}q^{\sigma(\l)}\equiv
q^{\sigma(\l)}\sum_{j=0}^{d-1}q^{j\cdot\frac{n}{d}\cdot{\bf y}(\k_1+\cdots+\k_{d})}
=
q^{\sigma(\l)}\cdot\frac{1-q^{d\cdot\frac{n}{d}\cdot{\bf y}(\k_1+\cdots+\k_{d})}}{1-q^{\frac{n}{d}\cdot{\bf x}(\k_1+\cdots+\k_{d})}}\equiv0\pmod{\Phi_n(q)}.
\end{align*}
Thus
$$
\sum_{\l\in \cQ_2}q^{\sigma(\l)}\equiv \sum_{\substack{\l\in \cQ_2\\ {\bf x}(\hat{\l})=0}}q^{\sigma(\l)}\pmod{\Phi_n(q)}.
$$

Similarly, by considering the action of $\psi$ on $\cQ_3$, we also can get
$$
\sum_{\l\in \cQ_3}q^{\sigma(\l)}\equiv \sum_{\substack{\l\in \cQ_3\\ {\bf x}(\hat{\l})=0}}q^{\sigma(\l)}\pmod{\Phi_n(q)}.
$$

Let us turn to $\cQ_4$. Unfortunately, $\psi$ is not a group action on $\cQ_4$, since there exists $\l\in \cQ_4$ such that $\psi_1(\l)\in \cQ_1$. So we must use a different type of group action.
For $\l\in \cQ_4$, write
$$
\bar{\l}+\hat{\l}=(\e_1+\v_1)+(\e_2+\v_2)+\cdots+(\e_n+\v_n),
$$
where for each $j$ all steps of $\v_j$ are $(0,1)$, and $\e_j$ is just formed by a single step which is $(1,0)$ or $(1,1)$.
That is, partition $\bar{\l}+\hat{\l}$ into the sum of some single east or northeast steps and some vertical paths (see the following graph).
\begin{center}
\includegraphics[width=0.3\textwidth]{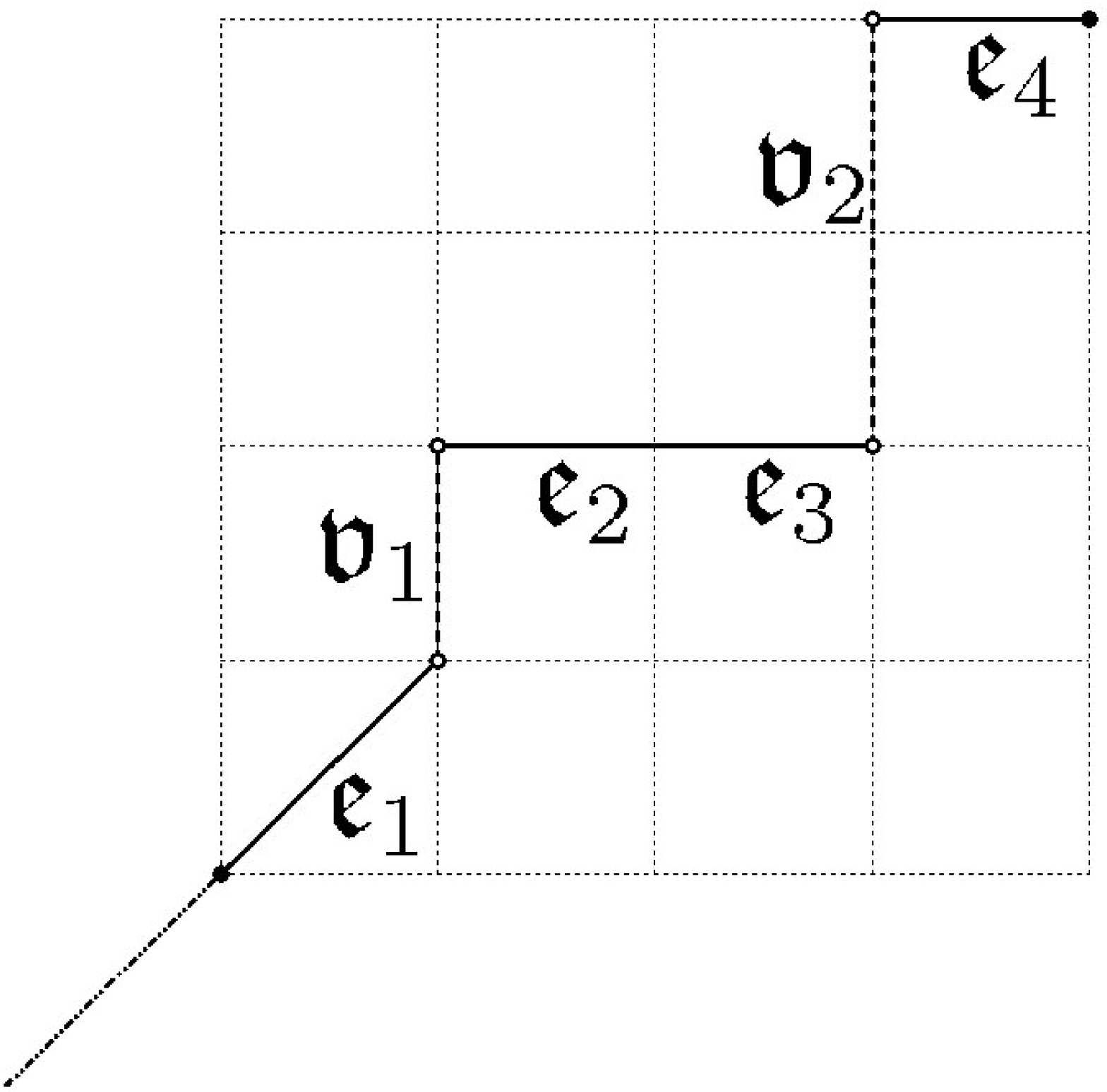}
\end{center}
Define
$$
\tau_1(\l)=\check{\l}+(\e_d+\v_1)+(\e_1+\v_2)+(\e_2+\v_3)+\cdots+(\e_{d-1}+\v_d).
$$
An example for $\tau_1$ is showed as follows.
\begin{center}
\includegraphics[width=0.6\textwidth]{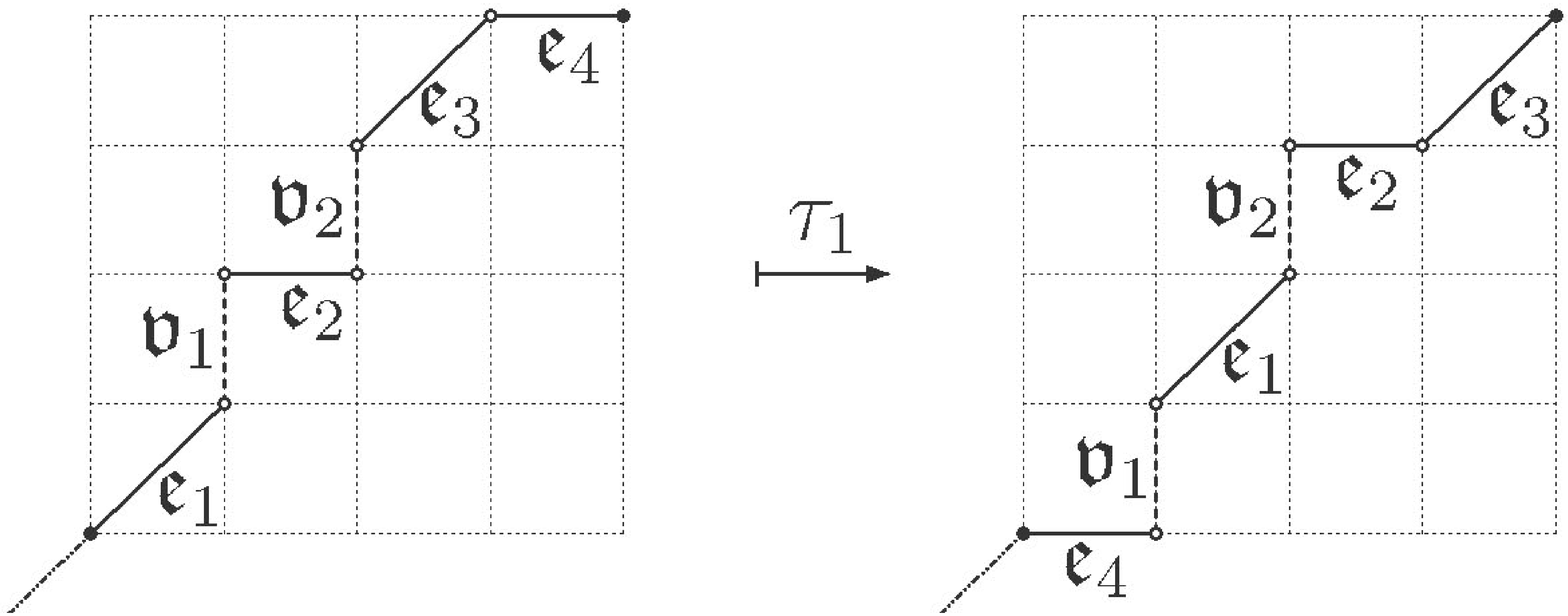}
\end{center}
Then $\tau_1$ also can generate a group action of $\Z_n$ on $\cQ_4$. We need to compute $\sigma(\tau_1(\l))-\sigma(\l)$. When $\l$ is transferred to $\tau_1(\l)$, for those steps of  $\v_1,\ldots,\v_h$, the $x$-coordinate of their endpoints will not vary. For $1\leq j\leq d-1$, if $\e_j=(1,1)$, then the $x$-coordinate of the endpoint of $\e_j$ will be added by $1$. And if $\e_d=(1,1)$, the $x$-coordinate of its endpoint will be subtracted by $n-1$. So letting $$s_\l=|\{j:\e_j=(1,1)\}|,$$ we have 
$$
\sigma(\tau_1(\l))-\sigma(\l)=\begin{cases}s_\l,&\text{if }\e_d=(1,0),\\
 (s_\l-1)-(n-1),&\text{if }\e_d=(1,1).\\
\end{cases}
$$
Thus we always have
$$
\sigma(\tau_1(\l))-\sigma(\l)\equiv s_\l\pmod{n}.
$$
Notice that clearly $|\O_{\tau,\l}|=1$ if and only all steps of $\bar{\l}+\hat{\l}$ are $(1,1)$, i.e., $s_\l=n$. 
Now assume that $|O_{\l}|=d>1$. Then
$$
\bar{\l}+\hat{\l}=(\e_1+\v_1)+\cdots+(\e_d+\v_d)+(\e_1+\v_{d+1})+\cdots+(\e_d+\v_{2d})+\cdots+
(\e_1+\v_{n-d+1})+\cdots+(\e_d+\v_{n}),
$$
So $s_\l$ must be  a multiple of $n/d$.
Since $1\leq s_\l<n$ now, we also have
$$
\sum_{\h\in\O_{\tau,\l}}q^{\sigma(\l)}\equiv
q^{\sigma(\l)}\sum_{j=0}^{d-1}q^{js_\l}
=
q^{\sigma(\l)}\cdot\frac{1-q^{ds_\l}}{1-q^{s_\l}}\equiv0\pmod{\Phi_n(q)}.
$$

Now we get
$$
\sum_{\l\in P_{h+n,k+n}}q^{\sigma(\l)}\equiv
\sum_{\substack{\l\in \cQ_1\\ {\bf y}(\hat{\l})=0}}q^{\sigma(\l)}+
\sum_{\substack{\l\in \cQ_2\\ {\bf x}(\hat{\l})=0}}q^{\sigma(\l)}+\sum_{\substack{\l\in \cQ_3\\ {\bf x}(\hat{\l})=0}}q^{\sigma(\l)}+
\sum_{\substack{\l\in \cQ_4\\ s_{\l}=n}}q^{\sigma(\l)}\pmod{\Phi_n(q)}.
$$
If $\l\in \cQ_1$ and ${\bf y}(\hat{\l})=0$, then
$
\sigma(\l)=\sigma(\check{\l}+\bar{\l})
$
since all steps of $\hat{\l}$ are $(1,0)$. So
$$
\sum_{\substack{\l\in \cQ_1\\ {\bf y}(\hat{\l})=0}}q^{\sigma(\l)}=
\sum_{\l\in P_{h+n,k}}q^{\sigma(\l)}=D_q(h+n,k).
$$
And
if $\l\in \cQ_2\cup \cQ_3$ and ${\bf x}(\hat{\l})=0$,
then
$$
\sigma(\l)=\sigma(\check{\l}+\bar{\l})+n(k+n).
$$
It follows that
$$
\sum_{\substack{\l\in \cQ_2\cup \cQ_3\\ {\bf x}(\hat{\l})=0}}q^{\sigma(\l)}=
q^{n(k+n)}\sum_{\l\in P_{h,k+n}}q^{\sigma(\l)}\equiv
\sum_{\l\in P_{h,k+n}}q^{\sigma(\l)}=D_q(h,k+n)\pmod{\Phi_n(q)}.
$$
Suppose that $\l\in \cQ_4$ and $s_\l=n$. Since $\bar{\l}$ just includes one point $(h,k)$ and all steps of $\hat{\l}$ are $(1,1)$,
we have
$$
\sigma(\l)=\sigma(\check{\l})+\sum_{j=1}^n(k+j)=\sigma(\check{\l})+kn+\frac{n+1}{2}.
$$
So
$$
\sum_{\substack{\l\in \cQ_4\\ s_\l=n}}q^{\sigma(\l)}=
q^{kn+\frac{n+1}{2}}\sum_{\l\in P_{h,k}}q^{\sigma(\l)}\equiv
q^{\frac{n+1}{2}}\sum_{\l\in P_{h,k+n}}q^{\sigma(\l)}=q^{\frac{n+1}{2}}D_q(h,k)\pmod{\Phi_n(q)}.
$$
If $n$ is odd, then
$$
q^{\frac{n+1}{2}}=(q^n)^{\frac{n+1}{2}}\equiv1\pmod{\Phi_n(q)}.
$$
Suppose that $n$ is even. Noting that
$$
1+q^{\frac n2}=\frac{1-q^n}{1-q^{\frac n2}}\equiv 0\pmod{\Phi_n(q)},
$$
we have
$$
q^{\frac{n+1}{2}}=(q^{\frac n2})^{n+1}\equiv(-1)^{n+1}=-1\pmod{\Phi_n(q)}.
$$
All are done.\qed

\end{document}